\documentclass[10pt]{article}
\usepackage[top=1in, bottom=1in, left=1in, right=1in]{geometry}

\usepackage{amsthm}
\usepackage{amssymb}

\usepackage{latexsym}
\usepackage{graphicx}
\usepackage{amsmath,graphicx}
\usepackage{amsfonts}
\usepackage{amssymb}
\usepackage{url}
\newtheorem{lemma}[]{Lemma}

\newtheorem{thm}[]{Theorem}

\newtheorem{cor}[]{Corollary}

\usepackage{epstopdf}

\usepackage{multicol}
\usepackage{algorithm,algorithmic}

\usepackage{fancyhdr}
\usepackage[us,12hr]{datetime} 
\fancypagestyle{plain}{
\fancyhf{}
\lfoot{Page \thepage}
}
\begin{document}

\title{Oriented Book Embeddings}

\author{Stacey McAdams \footnote{Louisiana Tech University, Ruston, Louisiana, 71272, USA, Email:  smcadams@latech.edu} \qquad  Jinko Kanno \footnote{Louisiana Tech University, Ruston, Louisiana, 71272, USA, Email: jkanno@latech.edu}
}

\date{February 3, 2016}
\maketitle

\begin{abstract} A graph $G$ has a {\it $k$-page book embedding} if $G$ can be embedded into a {\it $k$-page book}.  The minimum $k$ such that $G$ has a $k$-page book embedding is the {\it book thickness} of $G$, denoted $bt(G)$. Most of the work on this subject has been done for unoriented graphs and oriented acyclic graphs (no directed cycles). In this work we discuss oriented graphs $\overrightarrow{D}$ containing directed cycles by using {\it oriented book embeddings} and {\it oriented book thickness}, $obt(\overrightarrow{D})$. To characterize $\overrightarrow{D}$ such that $obt(\overrightarrow{D}) = k$, we define the class $\mathcal{M}^k$ of {\it $k$-page critical oriented graphs} to be all oriented graphs $\overrightarrow{D}$ with $obt(\overrightarrow{D}) =k$, but for every proper oriented subgraph of $\overrightarrow{D}$, denoted $\overrightarrow{D}'$, we have that $obt(\overrightarrow{D}') < k$.  Determining $\mathcal{M}^k$ for general $k$ is challenging; we narrow down the list of oriented graphs in $\mathcal{M}^k$ for small $k$. In this work we show complete lists for $\mathcal{M}^1$ and for $\mathcal{M}^2 \cap \mathcal{U}$, where $\mathcal{U}$ consists of all \textit{strictly dicyclic oriented graphs}, that is, oriented graphs containing exactly one oriented cycle, which is a directed cycle. 

 \end{abstract}

\section{Introduction}


In this work, every graph is a simple graph, i.e., a graph containing no loops or multiple edges.  We follow the notation and terminology in \cite{West}, in particular, we denote a complete graph on $n$ vertices as $K_n$, and a complete bipartite graph having two partite sets with sizes $m$ and $n$ as $K_{m,n}$. We consider orientations of simple graphs as described below. 

Given a graph $D$, with finite vertex set $V(D)$ and finite edge set $E(D)$, we define an {\it oriented graph} as follows. For each pair of vertices $u$ and $v$ in $V(D)$ such that there exists an edge $e = uv$ in $E(D)$, we assign $u$ (resp. $v$) to be the {\it tail} and $v$ (resp. $u$) to be the {\it head}. This assignment results in an {\it arc}, where the tail {\it is directed} to the head. We denote the arc $a$ with tail $u$ and head $v$ as $(u,v)$, and call $u$ and $v$ the {\it endpoints} of $a$. The endpoints of an arc are {\it adjacent} and an arc is {\it incident} to both of its endpoints. We denote the resulting oriented graph $\overrightarrow{D}$ with arc set $A(\overrightarrow{D})$ and vertex set $V(\overrightarrow{D})$, which is equal to $V(D)$. Given an oriented graph $\overrightarrow{D}$, the {\it underlying graph} $D$ is the graph such that the directions associated to the arcs of $\overrightarrow{D}$ are removed, yielding edges. An {\it oriented subgraph} of an oriented graph $\overrightarrow{D}$ is an oriented graph $\overrightarrow{D'}$ such that $V(\overrightarrow{D'}) \subseteq V(\overrightarrow{D})$, $A(\overrightarrow{D'}) \subseteq A(\overrightarrow{D})$ and for an arc $a = (u,v) \in A(\overrightarrow{D'})$, it must be true that $a \in A(\overrightarrow{D})$ and $ \{ u,v \} \subseteq V(\overrightarrow{D'})$. A \textit{proper oriented subgraph} of an oriented graph $\overrightarrow{D}$ is an oriented subgraph $\overrightarrow{D'}$ of $\overrightarrow{D}$ such that either $V(\overrightarrow{D'}) \subsetneqq V(\overrightarrow{D})$ or $A(\overrightarrow{D'}) \subsetneqq A(\overrightarrow{D})$.

Given a graph $G$, if two vertices $u$ and $v$ are adjacent in $G$, i.e., if there is an edge $e = uv \in E(G)$, we say that $u$ and $v$ are {\it neighbors} in $G$, and denote the neighborhood of a vertex $u$ as $N(u)$. For an oriented graph $\overrightarrow{D}$ containing an arc $(u,v)$, we say that $u$ is an {\it in-neighbor} of $v$ and that $v$ is an {\it out-neighbor} of $u$. For a vertex $u \in V(\overrightarrow{D})$, we denote the set of all in-neighbors of $u$ as $N^-(u)$, and the set of all out-neighbors of $u$ as $N^+(u)$. We also say that the {\it in-degree} of $u$ is $|N^-(u)|$ and the {\it out-degree} of $u$ is $|N^+(u)|$. The {\it neighborhood} of $u$ is defined to be $N(u) := N^+(u) \cup N^-(u)$, and the {\it degree} of $u$ is $|N^-(u) \cup N^+(u)|$. If $N(u) = N^-(u)$, then $u$ is said to be a {\it sink}; if $N(u) = N^+(u)$, then $u$ is said to be a {\it source}.  

For a graph $D$ with edge $e$, the {\it deletion of $e$} is an operation that yields a graph with edge set $E(D) \setminus e$ and vertex set $V(D)$. We denote the graph $D$ with edge $e$ deleted as $D \setminus e$. The deletion of an arc $a$ in an oriented graph $\overrightarrow{D}$ is defined analogously, and is denoted $\overrightarrow{D} \setminus a$. The {\it converse} of an arc $a= (u,v)$ is the arc $(v,u)$, denoted $a^*$. If we replace an arc $a$ with its converse $a^*$, we obtain an oriented graph with arc set $A(\overrightarrow{D} \setminus a) \cup a^*$ and vertex set $V(\overrightarrow{D})$; we denote the resulting oriented graph $\overrightarrow{D}(a^*)$. We call this replacement {\it switching the direction of a}. If we switch the direction of every arc in an oriented graph $\overrightarrow{D}$, we call the resulting oriented graph the \textit{converse} of $\overrightarrow{D}$ and denote it $\overrightarrow{D}^*$. For a graph $D$ with vertex $v$, the {\it deletion of v} is an operation that yields a graph with vertex set $ V(D) \setminus v$ and edge set $E(D) \setminus \{e_1, e_2,...,e_k \}$, such that $e_i$, $1 \le i \le k$, is incident to $v$. We denote the graph $D$ with vertex $v$ deleted as $D-v$. The deletion of a vertex $v$ in an oriented graph $\overrightarrow{D}$ is defined analogously and is denoted $\overrightarrow{D}-v$.

For $n \in \mathbb{N}$, let $[n]$ denote $ \{ 1,2,...,n \}$. For $n \geq 2$, the standard path, denoted $P_n$, is the graph with $V(P_n) = \{i \, \vert \, i \in [n] \}$ and $E(P_n)$ consisting of the edges with endpoints $i$, $i+1$ for $1 \le i \le n-1$. A {\it path on $n$ vertices} (resp. an {\it n-path}) is a graph isomorphic to $P_n$ (for some $n$). A graph is \textit{connected} if each pair of vertices belongs to a path; an oriented graph is connected if its underlying graph is connected. If a graph is not connected, it is \textit{disconnected} and its maximal connected subgraphs are called \textit{components}. The standard cycle, denoted $C_n$, is the graph $P_n$ with the added edge having endpoints $1$ and $n$. A {\it cycle on $n$ vertices} (resp. an {\it n-cycle}) is a graph isomorphic to $C_n$ (for some $n$). 

The  {\it standard directed path}, $\overrightarrow{P_n}$, is the oriented graph whose underlying graph is isomorphic to $P_n$, and for each edge $ \{i, i+1 \}$, we have the arc $(i, i+1)$. An oriented graph isomorphic to $\overrightarrow{P_n}$ is called an {\it n-dipath}. The {\it standard directed cycle}, denoted $\overrightarrow{C_n}$, is the oriented graph consisting of $\overrightarrow{P_n}$ and the arc $(n,1)$. An oriented graph isomorphic to $\overrightarrow{C_n}$ is referred to as an {\it n-dicycle}. One may notice that we specifically stated {\it directed} path or {\it directed} cycle. This is because the directions of the arcs ``follow'' each other along the path or the cycle. If this is not the case, and the underlying graph is isomorphic to a path or a cycle, but arbitrarily directed, we refer to each as an {\it oriented n-path} or {\it oriented n-cycle}, respectively. 

For an integer $k \geq 0$, a {\it $k$-book}, or {\it a book with $k$ pages}, consists of a line $L$ in $\mathbb{R}^3$, called the {\it spine}, where $L$ is identified with the $z$-axis, i.e.  $L = \{(0,0,z) : z \in \mathbb{R} \}$ and $k$ distinct closed half planes, called {\it pages}, whose common boundary is $L$. If $p$ is a page, then $p = \{(0,y,z) : y,z \in \mathbb{R}, y \geq 0 \}$.  The {\it interior} of a page $p$ is the open half plane, $p \setminus L$. A {\it $k$-page book embedding} is an embedding of a graph $G$ into a $k$-book such that:
\begin{itemize}
\item each vertex $v \in V(G)$ is embedded into $L$;
\item each edge $e \in E(G)$ is either completely embedded into $L$, or all points of $e$ (except the endpoints) are embedded into the interior of a single page. 

\end{itemize}

The {\it spine order}, for a particular book embedding, is the ordering of the vertices embedded into the spine whose $z$-coordinates are strictly increasing or decreasing. We can label the vertices embedded into the spine $v_1, v_2,..., v_n$. If two edges $v_iv_j$, with $i<j$ and $v_lv_m$, with $l<m$ are embedded into the same page, then it must be true that either $i<j<l<m$, $l<m<i<j$, $i<l<m<j$, or $l<i<j<m$. In this case, we say that the two edges follow the {\it planarity rule}. The {\it book thickness of a graph}, denoted $bt(G)$ is the minimum $k$ required for $G$ to have a $k$-page book embedding. 

For a particular book embedding of a graph $D$, we call an edge in $D$ that is embedded into the spine a {\it tight edge}, and we call an edge in $D$ that is not a tight edge a {\it loose edge}. In a given book embedding of a graph $D$, we call a vertex $v$ of $D$ a {\it tight vertex} if it is the common endpoint of exactly two tight edges; we call $v$ a {\it half-loose vertex} if it is the endpoint of exactly one tight edge, and we call $v$ a {\it loose vertex} if it is the endpoint of no tight edge. For a particular book embedding of a graph $D$, we say that an edge $uv$ in $D$ {\it covers} a vertex $x \in V(D)$ if $x$ is between $u$ and $v$ in the spine order of the embedding. For two arcs $a_1,a_2$ that are in the same page, if $a_1$ and $a_2$ share an endpoint, and $a_1$ covers the other endpoint of $a_2$, or if $a_1$ covers both endpoints of $a_2$, we say that $a_1$ and $a_2$ are \textit{nested} or that \textit{$a_2$ is nested inside $a_1$}.

A {\it planar graph} is a graph that has an embedding in the plane, and an {\it outerplanar graph} is a graph that has an embedding in the plane with every vertex on the unbounded face. A graph is {\it hamiltonian} if it contains a cycle that passes through every vertex of the graph.

\begin{thm}$($Bernhart and Kainen, $1979$ ~\cite{BK} $)$ Let $G$ be connected, then the following hold:
\noindent $1.$ $bt(G) = 0$ if and only if $G$ is a path, and \\
$2.$ $bt(G) \le 1$ if and only if $G$ is outerplanar. \\
$3.$ $bt(G) \le 2$ if and only if $G$ is a subgraph of a hamiltonian planar graph.
\label{BKthm}
\end{thm}

We call a graph $G$ a {\it $k$-page critical graph} if $bt(G) = k$, and for every edge $e \in E(G)$, we have that $bt(G \setminus e) = k-1$. By Kuratowski's Theorem~\cite{K} and Theorem~\ref{BKthm},  $K_4$ and $K_{2,3}$ are $2$-page critical graphs. All $k$-page critical graphs are connected; since, if $G$ was $k$-page critical with distinct components $H_1,H_2$, then  $bt(G) = max \{bt(H_1), bt(H_2) \} =k$. Assume $bt(G) = bt(H_1)$, and if we delete an edge $e' \in E(H_2)$, we have $bt(G \setminus e') =k$, and thus $G$ is not a $k$-page critical graph.

We now focus on embedding an oriented graph into a $k$-book. Since the spine $L$ of a $k$-book is identified with the $z$-axis, for convention we say that the spine is oriented upwards. We can then describe the orientation of the arcs embedded into the $k$-book as being {\it upwards} or {\it downwards}, relative to the spine. A {\it $k$-page oriented book embedding} for an oriented graph $\overrightarrow{D}$ is an embedding into a $k$-book such that:
\begin{itemize}
\item each vertex $v \in V(\overrightarrow{D})$ is embedded into $L$;
\item each arc $a \in A(\overrightarrow{D})$ is either completely embedded into $L$, or all points of $a$ (except the endpoints) are embedded into the interior of a single page. 
\item the orientation of all arcs embedded into $L$ agree;
\item the orientation of all arcs embedded into a given page agree. (We call this restriction the {\it direction rule}.)

\end{itemize}

If all arcs embedded into a single page are upwards, we refer to the page as an \textit{upwards page}, or simply say the \textit{page is upwards}; similarly if all arcs embedded into a single page are downwards, we refer to the page as an \textit{downwards page}, or simply say the \textit{page is downwards}. We do not restrict the direction of all arcs in the oriented book embedding to be the same, just those in the same page; an oriented book embedding may have an upwards page and a downwards page. The {\it oriented book thickness}, denoted $obt(\overrightarrow{D})$ is the minimum $k$ such that $\overrightarrow{D}$ has a $k$-page oriented book embedding. From the definition of oriented book thickness and its requirements, for an oriented graph $\overrightarrow{D}$ with underlying graph $D$, we have that $bt(D) \le obt(\overrightarrow{D})$. The definitions for spine order, tight or loose arcs, and whether an arc covers a vertex in a particular oriented book embedding are all similar to their definitions for book embeddings in the unoriented case.

We call an oriented graph $\overrightarrow{D}$ a {\it $k$-page critical oriented graph} if $obt(\overrightarrow{D}) = k$, and for every arc $a \in A(\overrightarrow{D})$, we have that $obt(\overrightarrow{D} \setminus a) = k-1$. The class of all $k$-page critical oriented graphs is denoted $\mathcal{M}^k$.





Determining $\mathcal{M}^k$ for general $k$ is challenging; we narrow down the list of oriented graphs in $\mathcal{M}^k$ for small $k$. In Section 3 we show complete list for $\mathcal{M}^1$, and in Section 4 we discuss $\mathcal{M}^2$. For undirected graphs, there are exactly two $2$-page critical graphs, $K_4$ and $K_{2,3}$; however, in the oriented case is becomes more complicated. We show a complete list for $\mathcal{M}^2 \cap \mathcal{U}$, where $\mathcal{U}$ consists of all \textit{strictly dicyclic oriented graphs}, that is, oriented graphs containing exactly one oriented cycle, which is a directed cycle. To prepare for the proofs in Section 4, we first discuss the oriented book embeddings of \textit{oriented cycles} and \textit{oriented trees} in Section 2.



\section{Oriented Book Embeddings of Cycles and Trees}
In this section we discuss the oriented book embeddings of two fundamental types of oriented graphs, oriented cycles and oriented trees. We will use these oriented graphs to construct strictly uni-dicyclic graphs in Section 4.

\subsection{Undirected and Oriented Cycles}
We now discuss oriented book embeddings of oriented cycles. We first discuss book embeddings of (undirected) cycles in Theorem~\ref{cyclic_on_C}, and characterize the spine order of every $1$-page book embedding of a cycle. We then describe all possible $1$-page oriented book embeddings of directed cycles in Corollary~\ref{directed cycle}. 

Let $C$ be a cycle on $n$ vertices. By definition, there exists an isomorphism $\phi: [n] \rightarrow V(C)$ such that the edge set of $C$ consists of $\phi(i) \phi(i+1)$ for $i \in [n-1]$ and edge $\phi(1) \phi(n)$. For $n \in \mathbb{N}$ and $k \in \mathbb{Z}$, a {\it cyclic permutation} of $V(C)$ is a mapping $\phi(i) \rightarrow \phi(i+k)$ $(mod$ $n)$, $i \in [n]$.  We call each cyclic permutation of $V(C)$ a {\it natural ordering} of $V(C)$.





\begin{thm}
\label{cyclic_on_C}
Let $C$ be a cycle having a $k$-page book embedding. Then $k$ is minimal if and only if the spine order is a natural ordering of $V(C)$. 
\end{thm}



\begin{proof}
Since $C$ is not a path, by Theorem~\ref{BKthm}, $bt(C) \geq 1$ and since $C$ is outerplanar, $bt(C) \leq 1$. Therefore $k$ is minimal when $k=1$. We need only prove the necessary condition, as the sufficient condition is trivial. 

To prove the necessary condition by contradiction, choose a counterexample, that is, suppose that there exists a 1-page book embedding of $C$ such that $(1)$ the spine order is not a natural ordering of $V(C)$; then there exists at least one pair of vertices, $\phi(i)$, $\phi(i+1)$ in $C$ that are not consecutive in the spine, implying that the edge $\phi(i) \phi(i+1)$ is a loose edge. Considering all such pairs, $(2)$ choose $i$ such that no other such pair exists between $\phi(i)$ and $\phi(i+1)$ in the spine. By $(1)$, there is at least one vertex $t$ embedded between $\phi(i)$ and $\phi(i+1)$ in the spine. Since $t$ is covered by the edge $\phi(i) \phi(i+1)$, by the planarity rule, vertex $t+1$ or vertex $t-1$ must also be covered by $\phi(i) \phi(i+1)$. If $\{t+1, t-1 \} = \{ \phi(i), \phi(i+1) \}$, then $n =3$; every spine order of three vertices is a natural ordering, and we obtain a contradiction. Therefore, either $t+1$ or $t-1$ is not in $\{ \phi(i), \phi(i+1) \}$. By continuing this logic for an arbitrary vertex $t$, every vertex except $\phi(i)$ and $\phi(i+1)$ is embedded between $\phi(i)$  and $\phi(i+1)$; otherwise, the planarity rule is violated. By $(2)$, the spine order must be a natural ordering of $V(C)$. \end{proof}



Corollary~\ref{directed cycle} follows from Theorem~\ref{cyclic_on_C} and provides the location of the arcs in a $1$-page oriented book embedding of a directed cycle.

\begin{cor}
\label{directed cycle}
Let $\overrightarrow{D}_n$ be an $n$-dicycle. Then $\overrightarrow{D}_n$ has a $1$-page oriented book embedding  if and only if the spine ordering is a natural ordering of $V(\overrightarrow{C})$ and every arc is tight, except the arc between the top vertex and bottom vertex in the spine.
\end{cor}

We now show that every oriented cycle has a $1$-page oriented book embedding. 

\begin{thm}
Let $\overrightarrow{C}$ be an oriented cycle. Then $obt(\overrightarrow{C}) =1$.
\end{thm}

\begin{proof}
Let $C$ be the underlying graph of $\overrightarrow{C}$. Since $bt(C) \geq 1$, we have that $obt(\overrightarrow{C}) \geq 1$. To show that $obt(\overrightarrow{C}) \le 1$, embed the vertices into the spine so that the spine order is a natural ordering of $V(\overrightarrow{C})$, say $\phi(1), \phi(2),..., \phi(n)$. Embed the arc $a$ with endpoints $\phi(1)$,$\phi(n)$ into the page. For $1 < i \le n-1$, embed every arc having the same direction as $a$ into the page and embed each arc having the opposite direction of $a$ into the spine. Thus we obtain a 1-page oriented book embedding, and $obt(\overrightarrow{C}) \le 1$. Since we have shown that $1 \le obt(\overrightarrow{C}) \le 1$, we conclude that $obt(\overrightarrow{C}) = 1$.
\end{proof}



\subsection{Oriented Trees}
We now discuss oriented book embeddings of oriented trees and oriented forests, and introduce a type of oriented tree, called a {\it fountain tree}, which will be utilized in Section 4. A {\it tree}, denoted $T$, is a connected graph having no cycle as a subgraph. An {\it oriented tree}, denoted $\overrightarrow{T}$, is an oriented graph whose underlying graph is a tree.
 
The authors in~\cite{SQLptI}, focusing on oriented graphs which contain no directed cycle, use a restricted definition of oriented book embeddings, requiring the direction of all arcs in every page to agree and such that no arc is embedded into the spine, i.e., each vertex \textit{must} be loose in the embedding. They prove that every oriented tree has a $1$-page oriented book embedding such that an arbitrarily chosen vertex, called the \textit{root} is uncovered. We translate their result below.

\begin{thm} \label{oriented tree} $($Heath, Pemmaraju, and Trenk, $1999$ ~\cite{SQLptI} $)$  For every oriented tree $\overrightarrow{T}$ with arbitrarily chosen root $v$, there exists a spine order of $V(\overrightarrow{T})$ that yields a $1$-page oriented book embedding of $\overrightarrow{T}$ in which the root $v$ is uncovered, every vertex is loose, and the page is upwards.   \end{thm}



The above result only proves existence. We give an algorithm, called the \underline{O}riented \underline{T}ree \underline{S}pine \underline{O}rder Algorithm, or OTSO Algorithm, that gives a spine order which yields a $1$-page oriented book embedding as described in Theorem~\ref{oriented tree}. For convenience, we will always assume the direction of the page to be upwards, unless otherwise noted. For an oriented tree $\overrightarrow{T}$, we can choose the root to be a sink in $\overrightarrow{T}$. The following lemma proves that, in a $1$-page oriented book embedding of $\overrightarrow{T}$, we can place the root so that it is the top vertex in the spine.

\begin{algorithm}
\small
\caption{OTSO Algorithm}
\label{alg1}

\begin{algorithmic}[1]
  \REQUIRE An oriented tree $\overrightarrow{T}$, an arbitrarily fixed vertex $x \in V(\overrightarrow{T})$, and two empty lists $L$ and $S$
  \ENSURE A list $S$ that yields the spine order of a $1$-page oriented book embedding of $\overrightarrow{T}$ such that each vertex is loose in the embedding, the direction of all arcs agree, and $x$ is uncovered
\begin{multicols}{2}
  \WHILE{$L \neq \emptyset$}  
  
  \IF{$\vert N^+(x) \vert = n > 0, n \in \mathbb{N}$}
  \STATE Let $N^+(x) = \{u_1, u_2,...,u_n \}$
  \STATE Add $u_1$ to the end of $L$
  \STATE Add $u_1$ to the beginning of $S$

  		\IF{$n \geq 2$}
		
   			\FOR {$i = 2, i \leq n, i++$}
   			\IF {$u_i \notin L$}
        	\STATE add $u_i$ to $L$ between $u_{i-1}$ and $x$	 		        
        	\STATE add $u_i$ to $S$ between $u_{i-1}$ and $x$
        	\ENDIF
   			\ENDFOR
   			
  		\ENDIF

    \ENDIF  
 
  \IF{$\vert N^-(x) \vert = m > 0, m \in \mathbb{N}$}
  \STATE Let $N^-(x) = \{v_1, v_2,...,v_n \}$
  \STATE Add $u_v$ to the end of $L$
  \STATE Add $v_1$ to the end of $S$

  		\IF{$m \geq 2$}
		
   			\FOR {$i = 2, i \leq m, i++$}
   			\IF {$v_i \notin L$}
        	\STATE add $v_i$ to $L$ between $v_{i-1}$ and $x$  	 		        \STATE add $v_i$ to $S$ between $v_{i-1}$ and $x$
        	\ENDIF
   			\ENDFOR
   			
  		\ENDIF

    \ENDIF 
 \ENDWHILE 
 \end{multicols} 
\end{algorithmic}

\end{algorithm}

\begin{lemma} \label{tree_sink}  If a vertex $x$ is a sink in $\overrightarrow{T}$ then there exists a spine order of $V(\overrightarrow{T})$ which yields a $1$-page oriented book embedding of $\overrightarrow{T}$ such that $x$ is the top vertex in the spine and each vertex is loose in the oriented book embedding.\end{lemma}

\begin{proof} Since $x$ is a sink, let $N(x)=N^-(x)= \{v_1,v_2,...,v_k \}$. Delete $x$ to obtain $k$ disconnected oriented trees, denoted $\overrightarrow{T_{v_i}}$, with root $v_i$, $1 \le i \le k$. Let $\beta_i$ be the spine order of $\overrightarrow{T_{v_i}}$ in the $1$-page oriented book embedding guaranteed by Theorem~\ref{oriented tree}. To obtain a $1$-page oriented book embedding of $\overrightarrow{T}$, embed the vertices into the spine with spine order $(x, \beta_1, \beta_2,...,\beta_k)$ such that $x$ is the top vertex in the spine. By Theorem~\ref{oriented tree}, every arc of $\overrightarrow{T_{v_i}}$, $1 \le i \le k$, can be placed into the interior of the page such that $v_i$ is uncovered, each arc is upwards, and each vertex is loose. Then place each arc $(v_i,x)$ into the interior of the page, and since all arcs $(v_i,x)$ share $x$ as a common endpoint, they are nested and thus do not cross in the page. Therefore the statement holds. \end{proof}

Let $\overrightarrow{T}$ be an oriented tree with sink $x$. For convenience, we call an oriented book embedding, as described in Lemma~\ref{tree_sink}, a \textit{sink oriented book embedding}; we abbreviate the spine order of such an embedding $(x;\alpha_x)$, such that $x$ is the top vertex in the spine, and $\alpha_x$ represents $( \beta_1, \beta_2,...,\beta_k)$. Similarly, we can construct a \textit{source oriented book embedding}.

We are now prepared to define a sink fountain tree having an specific $1$-page oriented book embedding, which will be useful in Section 4. To do this, we use the standard dipath $\overrightarrow{P_n}$, with arcs $(i,i+1)$, $1 \le i \le n-1$, and $n$ oriented trees $\overrightarrow{T_{x_i}}$, $1 \le i \le n$, such that each $\overrightarrow{T_{x_i}}$ contains a sink $x_i$. By Lemma~\ref{tree_sink}, each $\overrightarrow{T_{x_i}}$ has a $1$-page sink oriented book embedding, such that $x_i$ is the top vertex in the spine, with spine order $(x_i, \alpha_i)$. Embed $V(\overrightarrow{P_n})$ and $V(\overrightarrow{T_{x_i}})$, $1 \le i \le n$, into the spine of a book with spine order $(1,2,...,n,(x_n; \alpha_n), (x_{n-1}; \alpha_{n-1}),...,(x_1; \alpha_1))$ such that $1$ is the top vertex in the spine. Place all arcs in $A(\overrightarrow{P_n})$ into the spine, and place all arcs in $A(\overrightarrow{T_{x_i}})$, $1 \le i \le n$, into the interior of the page. Therefore we have a $1$-page oriented book embedding $\overrightarrow{P_n} \cup \bigcup_{i=1}^n \overrightarrow{T_{x_i}}$. We now identify each $x_i$ in $\overrightarrow{T_{x_i}}$ with $i$ in $\overrightarrow{P_n}$, changing the spine order to $(x_1=1,x_2=2,...,x_n=n, \alpha_n, \alpha_{n-1},..., \alpha_1)$, as shown in Figure~\ref{fig: fnt_tree}
for $n=2$. This does not increase the number of pages required, since for $1 \le i \le n$, each set of arcs with head $x_i$ will be nested inside the set of arcs with head $x_{i-1}$. This yields an oriented tree containing an $n$-dipath. We call such an oriented tree a \textit{sink fountain tree} and call the $1$-page oriented book embedding described above a \textit{sink fountain oriented book embedding}. We denote the spine order of a such an embedding $(x_1; \gamma_{x_1})^f$, where $x_1$ is the top vertex in the spine and $\gamma_{x_1}$ represents $(x_2,...,x_n, \alpha_n, \alpha_{n-1},..., \alpha_1)$. Similarly, we can construct a \textit{source fountain tree}.

\begin{figure}[hbtp]
\centering
\includegraphics[scale=.4]{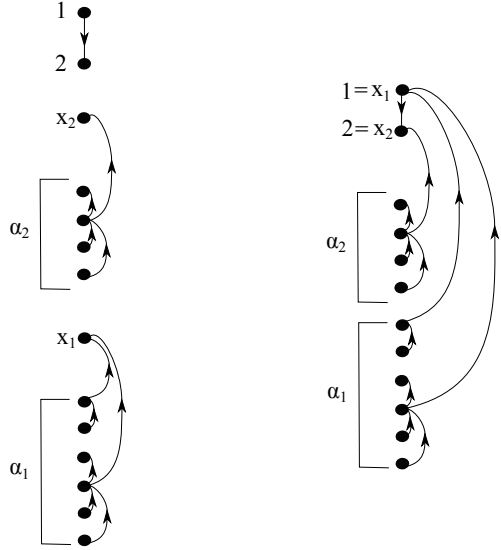}
\caption{Constructing a sink fountain tree}
\label{fig: fnt_tree}
\end{figure}

\section{$1$-page Critical Oriented Graphs}
We now characterize the class, $\mathcal{M}^1$, of $1$-page critical oriented graphs. The only connected oriented graph on $n$ vertices which can be embedded into a $0$-page book, i.e., only the spine, is a directed path $\overrightarrow{P_n}$. Since $\overrightarrow{P_n}$ has exactly one source and exactly one sink, to find minimal obstructions, we consider oriented graphs having a source or sink of degree at least two. Thus we define $S^+$ to be the oriented $3$-path containing a source of degree two, and we define its converse $(S^+)^*$ to be $S^-$, which contains a sink of degree two. If an oriented graph is an oriented tree that is not an oriented path, it contains a vertex of degree greater than or equal to three and must also contain an oriented subgraph isomorphic to either $S^+$ or $S^-$; therefore, the only other oriented graphs to consider are oriented cycles. If an oriented cycle is not a directed cycle, it must contain an oriented subgraph that is isomorphic to either $S^+$ or $S^-$. While an $n$-dicycle, denoted $\overrightarrow{D}_n$, contains no $S^+$ or $S^-$, its oriented book thickness is one, which will be discussed more in Section 2. However, the deletion of any arc in $\overrightarrow{D}_n$ results in a dipath; therefore, $\overrightarrow{D}_n$, $n \geq 3$, is in $\mathcal{M}^1$, giving the following result.

\begin{lemma} \label{M_1} $\mathcal{M}^1 = \{ S^+, S^-, \overrightarrow{D}_n \vert n \geq 3\}$. \end{lemma}

\section{Strictly Uni-dicyclic Critical Graphs}
In this section, we give the complete description for the class of \textit{strictly uni-dicyclic graphs}, denoted $\mathcal{U}$, and discuss three subclasses, $\mathcal{I},\mathcal{T},\mathcal{R}$, of $\mathcal{U}$ that are $2$-page critical, which we use to characterize $\mathcal{M}^2 \cap \mathcal{U}$ in Theorem~\ref{unicyclic_list}.  

Let $\overrightarrow{D}_n$, $n \geq 3$, be an $n$-dicycle, with arcs $(\phi(i),\phi(i+1))$, for $1 \le i \le n-1$, and arc $(\phi(n),\phi(1))$; then $\overrightarrow{D}_n$ is a member of $\mathcal{U}$ and we can construct every oriented graph in $\mathcal{U}$ as follows. Following the definition of $1$-sum for undirected graphs, found in~\cite{sum}, we define the {\it $1$-sum} of two oriented graphs $\overrightarrow{D}$, $\overrightarrow{D}'$, via $y$ and $y'$, to be the oriented graph obtained by identifying a vertex $y \in V(\overrightarrow{D})$ with a vertex $y' \in V(\overrightarrow{D}')$; we denote the resulting oriented graph $\overrightarrow{D}(y) +_1 \overrightarrow{D'}(y')$. For convenience, we call a single vertex a \textit{trivial oriented tree}. For $n$ distinct, possibly trivial, oriented trees $\overrightarrow{T_i}$, $1 \le i \le n$, we describe a member of $\mathcal{U}$ to be $\bigcup_{i=1}^n (\overrightarrow{D}_n(\phi(i)) +_1 \overrightarrow{T_i}(y_i))$, where $\phi(i) \in V(\overrightarrow{D}_n)$ and $y_i \in V(\overrightarrow{T_i})$. For an oriented graph $\overrightarrow{D}$ in $\mathcal{U}$, if $\overrightarrow{T_i}$ is not a trivial oriented tree, we call the vertex $\phi(i)=y_i$ in $\overrightarrow{D}$ a \textit{heavy vertex}. In fact, each member of $\mathcal{I}$, $\mathcal{T}$, and $\mathcal{R}$ have most three heavy vertices. The next result shows that if a strictly uni-dicyclic graph has exactly one heavy vertex, then its oriented book thickness is one. 




 \begin{lemma} \label{dicycle_1_tree}
 Let $\overrightarrow{D} \in \mathcal{U}$. If $\overrightarrow{D}$ has exactly one heavy vertex, then there exists a $1$-page oriented book embedding of $\overrightarrow{D}$ such that the page is upwards and the endpoints of the loose arc of $\overrightarrow{D_n}$ are half-loose in the embedding. 
 \end{lemma}
 
\begin{proof}
Let $\overrightarrow{D}_n$ be an $n$-dicycle with arcs $(\phi(i), \phi(i+1))$, for $1 \le i \le n-1$, and arc $(\phi(n),\phi(1))$. We may assume the heavy vertex in $\overrightarrow{D}$ is $\phi(1)$. Then there exists a non-trivial oriented tree $\overrightarrow{T}$, containing a vertex $y$, such that $\overrightarrow{D}= \overrightarrow{D_n}(\phi(1)) +_1 \overrightarrow{T}(y)$. Applying Theorem~\ref{oriented tree}, embed $\overrightarrow{T}$ into a $1$-page book such that each vertex is loose and the direction of the page is upwards. Since $y= \phi(1)$ is loose in the spine, we can insert the directed path $\phi(1), \phi(2),...,\phi(n)$ into the spine below $y$, so that each arc of the dipath is downwards in the spine, and place the arc $(\phi(n),\phi(1))$ into the interior of the page. Thus we have a $1$-page oriented book embedding of $\overrightarrow{D}$ with $\phi(1)$ and $\phi(n)$ half-loose in the embedding. 
\end{proof}

We now define the first subclass of strictly uni-dicyclic graphs, $\mathcal{I}$, as follows. An oriented graph $\overrightarrow{D}$ is a member of $\mathcal{I}$ if $\overrightarrow{D}$ contains an $n$-dicycle, $\overrightarrow{D_n}$, with $n \geq 4$, and for two vertices $\phi(i), \phi(j) \in V(\overrightarrow{D_n})$ with $\vert i - j \vert > 1$, there exist two arcs $a_i,a_j \in A(\overrightarrow{D})$, having one endpoint of degree one and other endpoint $\phi(i), \phi(j)$, respectively, in $\overrightarrow{D_n}$, such that $\overrightarrow{D} = \overrightarrow{D_n} \cup \{a_i,a_j \}$. Three members of $\mathcal{I}$, with $n =4$, are shown in Figure~\ref{fig:I}.

\begin{lemma}
\label{I>1}
If $\overrightarrow{D} \in \mathcal{I}$, then $obt(\overrightarrow{D}) > 1$.
\end{lemma}

\begin{proof} Assume there exists a $1$-page oriented book embedding of $\overrightarrow{D}$. Since $\overrightarrow{D}$ contains dicycle, by Lemma~\ref{directed cycle}, each vertex of $\overrightarrow{D}_n$ is tight in the embedding, except the endpoints of the loose arc. However $\phi(i)$ and $\phi(j)$ are not adjacent, thus at least one of $\{\phi(i),\phi(j)\}$, say $\phi(i)$, is covered by the loose arc of $\overrightarrow{D_n}$; let $u_i$ be the other endpoint of the $a_i$. Since $u_i$ must be located in the spine above all vertices of the dicycle, or below all vertices of the dicycle, $a_i$ crosses the loose arc of the dicycle, a contradiction to the planarity rule.  \end{proof}

\begin{figure}[hbtp]
\centering
\includegraphics[scale=.3]{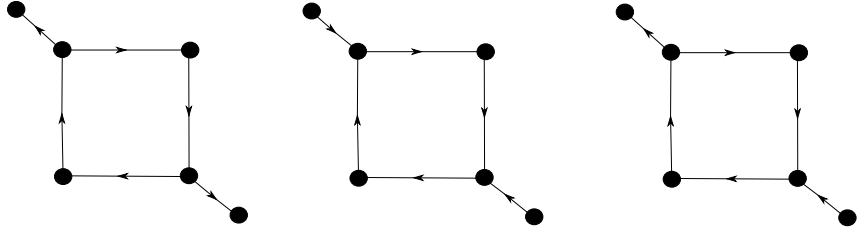}
\caption{Members of $\mathcal{I}$ (based on a $4$-dicycle)}
\label{fig:I}
\end{figure}

For members of the class $\mathcal{I}$, the length of the directed cycle is required to be at least four. Therefore, we define the class $\mathcal{T}$ to be the class of all oriented graphs $\overrightarrow{D}$ which contain a $3$-dicycle and such that there exist three arcs $a_1,a_2,a_3 \in A(\overrightarrow{D})$ having one endpoint of degree one and other endpoint $\phi(1), \phi(2), \phi(3)$, respectively, in $\overrightarrow{D_n}$, with $\overrightarrow{D} = \overrightarrow{D_n} \cup \{a_1,a_2,a_3 \}$. As with $\mathcal{I}$, the direction of each arc not contained in the directed cycle is not unique. Since the size of the dicycle in elements of $\mathcal{T}$ is restricted, we are able to easily list, in Figure~\ref{fig:T}, the four members of $\mathcal{T}$. 

\begin{lemma}
\label{T}
If $\overrightarrow{D} \in \mathcal{T}$, then $obt(\overrightarrow{D}) > 1$.
\end{lemma}

\begin{figure}[hbtp]
\centering
\includegraphics[scale=.4]{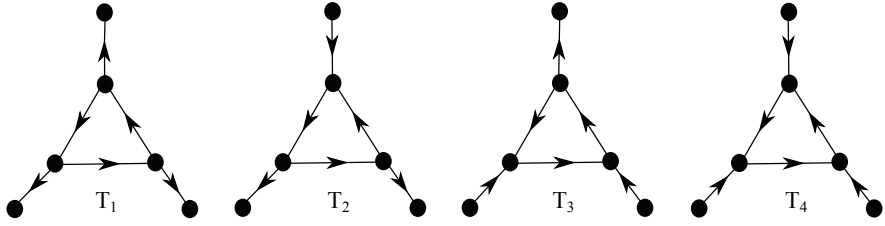}
\caption{Members of $\mathcal{T}$}
\label{fig:T}
\end{figure}

To introduce the last class, $\mathcal{R}$, of strictly uni-dicyclic graphs, we first define an oriented tree, called an \textit{antler}, using the oriented paths $S^+, S^-,$ and a standard $j$-dipath $\overrightarrow{P_j}$.  Let $s^+, s^-$ be the vertices of degree two in $S^+$, $S^-$, respectively. For an integer $j \geq 2$, we define a \textit{positive $j$-antler}, $\overrightarrow{A}_j^+ := S^+(s^+) +_1 \overrightarrow{P_j}(j)$; we also define the positive $1$-antler $A^+_1 := S^+$. We define a \textit{negative $j$-antler} $A^-_j$ to be the converse, $(A^+_j)^*$, of a positive $j$-antler. $A^+_3$ and $A^-_4$ can be seen in Figure~\ref{fig:antler}. It is important to note that if an oriented tree contains \textit{no} positive antler, it is a negative fountain tree, and similarly, if an oriented tree contains \textit{no} negative antler, it is a positive fountain tree. 

\begin{figure}[hbtp]
\centering
\includegraphics[scale=.4]{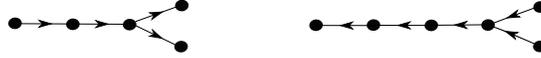}
\caption{$A^+_3$ and $A^-_4$}
\label{fig:antler}
\end{figure}

We now define $\mathcal{R}$, using a positive antler $A^+_j$, a negative antler $A^-_k$, and the standard $n$-dicycle, $\overrightarrow{D_n}$. In $\overrightarrow{D_n}$, we choose an arbitrary arc $(\phi(i), \phi(i+1))$ and call it $(x,y)$. Let the oriented graph $\overrightarrow{R_n}(j^+) := A^+_j(1) +_1 \overrightarrow{D_n}(x)$; let the oriented graph $\overrightarrow{R_n}(j^+,k^-):= A^-_k(k) +_1 \overrightarrow{R_n}(j^+)(y)$. We then define the class $\mathcal{R}$ as $\mathcal{R} := \{ \overrightarrow{R_n}(j^+,k^-) : n,j,k \in \mathbb{Z}^+, n \geq 3 \}$. Figure~\ref{fig:R} depicts $\overrightarrow{R_3}(1^+,1^-)$, $\overrightarrow{R_3}(2^+,1^-)$, and $\overrightarrow{R_3}(2^+,2^-)$.

\begin{figure}[hbtp]
\centering
\includegraphics[scale=.35]{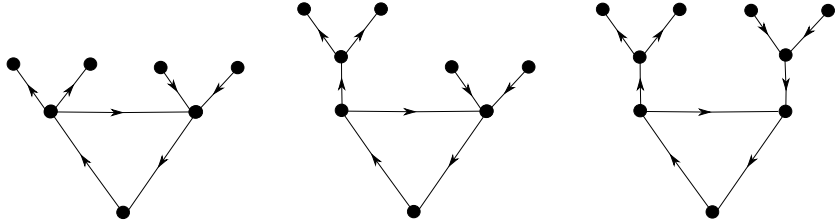}
\caption{Members of $\mathcal{R}$}
\label{fig:R}
\end{figure}

\begin{lemma} \label{R} Let $\overrightarrow{D} \in \mathcal{R}$. Then $obt(\overrightarrow{D}) > 1$. \end{lemma}

\begin{proof} Let $\overrightarrow{D} \in \mathcal{R}$. For a contradiction, consider a $1$-page oriented book embedding of $\overrightarrow{D}$. We will first prove the statement for $\overrightarrow{D}= \overrightarrow{R_n}(1^+,1^-)$. In $\overrightarrow{R_n}(1^+,1^-)$, we have that $s^+= x$ and $s^- = y$. Let $a_1, a_2$ be the sinks of the positive antler, and let $b_1, b_2$ be the sources of the negative antler. By Corollary~\ref{directed cycle}, the spine order of the embedding must be a natural ordering of $V(\overrightarrow{D_n})$ and each arc except one, call it $l$, is tight in the embedding and each vertex in $V(\overrightarrow{D_n})$, except the endpoints of $l$, is covered by $l$ in the embedding. Therefore, no vertex \underline{not} contained in $V(\overrightarrow{D_n})$ can appear between two vertices of $\overrightarrow{D_n}$ in the spine and $l = (x,y)$; otherwise, one of $x$ or $y$, say $x$, would be covered by $l$ and the arcs  $(x,a_1), (x,a_2)$ would cross $l$. We may assume then that $x$ is below $y$ in the spine, and the direction of the page is upwards. Both $a_1$ and $a_2$ cannot appear below $x$ in the spine, otherwise, say if $a_1$ is below $a_2$ is below $x$ in the spine, then we would have an upwards arc $(x,y)$ and a downwards arc $(x,a_1)$ in the interior of the page, a contradiction to the direction rule. Therefore, since each arc between $x$ and $y$ is tight, at least of $a_1,a_2$, say $a_2$ is above $y$ in the spine. Then $y$ is covered by $(x,a_2)$. Then both $b_1$ and $b_2$ must appear between $y$ and $a_2$ in the spine, thus either $(b_1, y)$ or $(b_2,y)$ is embedded into the page, as a downwards arc. However, since $(x,y)$ is an upwards arc, we obtain a contradiction and $obt(\overrightarrow{D}) > 1$. The proofs for other members of $\mathcal{R}$ are similar. \end{proof}



\newpage
The next theorem shows that $\mathcal{M}^2 \cap \mathcal{U}$, the class of $2$-page critical, strictly uni-dicyclic graphs, is completely characterized by the three classes of $2$-page critical graphs.


\begin{thm} \label{unicyclic_list} Let $\overrightarrow{D} \in \mathcal{U}$. Then $obt(\overrightarrow{D}) \le 1$ if and only if $\overrightarrow{D}$ contains no member of $\mathcal{T}$, $\mathcal{I}$, or $\mathcal{R}$; in other words, $\mathcal{M}^2 \cap \mathcal{U} =  \mathcal{T} \cup \mathcal{I} \cup \mathcal{R} $. \end{thm}

\begin{proof} To prove the necessary condition, consider the contrapositive, which states: If $\overrightarrow{D}$ contains a member of $\mathcal{I}$, $\mathcal{T}$, or $\mathcal{R}$ then $obt(\overrightarrow{D}) > 1$. This is true by Corollary~\ref{I>1}, Corollary~\ref{T}, and Lemma~\ref{R}. To prove the sufficient condition, let $\overrightarrow{D}$ be a strictly uni-dicylic graph containing no member of $\mathcal{I}$, $\mathcal{T}$, or $\mathcal{R}$, and we now construct a $1$-page oriented book embedding of $\overrightarrow{D}$. Let $\overrightarrow{D}$ contain an $n$-dicycle $\overrightarrow{D}_n$. Since $\overrightarrow{D}$ contains no member of $\mathcal{T}$ or $\mathcal{I}$ as an oriented subgraph, $\overrightarrow{D_n}$ has at most two heavy vertices $x$ and $y$, such that the arc $(x,y) \in A(\overrightarrow{D_n})$.  

Since $\overrightarrow{D}$ contains no member of $\mathcal{R}$, either $\overrightarrow{T_x}$ does not contain a positive antler, or $\overrightarrow{T_y}$ does not contain a negative antler. We may assume that $\overrightarrow{T_x}$ does not contain a positive antler. We first consider a $1$-page oriented book embedding of $\overrightarrow{D_n} \cup \overrightarrow{T_y}$, guaranteed by Lemma~\ref{dicycle_1_tree}. Since $x$ has degree two in $\overrightarrow{D_n} \cup \overrightarrow{T_y}$, $x$ is half-loose in the spine. We now put $\overrightarrow{T_x}$ back into $\overrightarrow{D}$. Since $\overrightarrow{T_x}$ contains no positive antler, $|N^+(x)| \le 1$. If $|N^+(x)| =0$, then $x$ is a sink, and by Lemma~\ref{tree_sink}, there is sink oriented book embedding of $\overrightarrow{T_x}$ with spine order $(x; \alpha_x)$. Thus place $\alpha_x$ below $x$ in the spine and we are done. If $|N^+(x)| =1$, then $\overrightarrow{T_x}$ must be a sink fountain tree, and there is a sink fountain oriented book embedding of $\overrightarrow{T_x}$ with spine order $(x; \gamma_{x})^f$. Thus place $\gamma_{x}$ below $x$ in the spine and we are done. In either case, we obtain a $1$-page oriented book embedding of $\overrightarrow{D}$. The case in which $\overrightarrow{T_y}$ does not contain a negative antler is similar.

 \end{proof}

\end{document}